\documentclass[10pt]{article}
\usepackage{amssymb}
\usepackage{mathtools}
\usepackage{hyperref}
\usepackage{amsmath,amsthm,amssymb}
\usepackage{mathtext}
\usepackage[T1,T2A]{fontenc}
\usepackage[utf8]{inputenc}
\usepackage[english]{babel}
\usepackage{indentfirst}
\usepackage{titlesec}
\titleformat*{\section}{\normalsize \bfseries}
\newtheorem{theorem}{Theorem}
\newtheorem{lemma}{Lemma}

\newcommand{\beq}{\begin{equation}}
	\newcommand{\eeq}{\end{equation}}
\newcommand{\Cc}{\mathbb{C}}

\newcommand{\Pp}{\mathbb{P}}

\newcommand{\ra}{\rightarrow}

\newcommand{\N}{\mathbb{N}}

\newcommand{\ir}{\textnormal{Int}}
\newcommand{\Zz}{\mathfrak{Z}}

\theoremstyle{definition}
\newtheorem{definition}{Definition}[section]

\begin{document}

		\title{Almost optimal polynomial approximation on convex sets in $\mathbb{C}$}
		\author{Liudmyla Kryvonos}
		\date{}
		\maketitle 
		\thispagestyle{empty}	
		
		\begin{abstract} 
			For a function $f$, continuous on a compact convex set $K$ and analytic in its interior we construct a sequence of almost optimal polynomials that converge with a geometric rate at points of analyticity of $f$.
			
		\end{abstract}
	
	\textit{\small MSC:} {\small 30E10 }
	
	\textit{\small Keywords:} {\small Polynomial approximation, Near-best approximation, Convex sets}
	
	\section{Introduction}
	
	Let $K\subset \mathbb{C}$ be a compact set with connected complement $\Omega:=\mathbb{\overline{C}}\backslash K$, where $\overline{\mathbb{C}}:=\mathbb{C}\cup \{\infty\}$ is the extended complex plane. The set of continuous functions on $K$ that are analytic in the interior $\ir K$ of $K$ will be denoted by $A(K)$.
	Let $\Pp_{n}$, $n\in \mathbb{N}_{0}:=\{0,1,2,...\}$ be the class of complex polynomials of degree at most $n$.
	
	The rate of the best uniform approximation of a function $f \in A(K)$ by polynomials of degree  at most $n$ we denote by
	\beq \label{best}
	E_{n}(f) = E_{n}(f, K) := \underset{P_{n} \in \Pp_{n}}{\inf}\|f - P_{n}\|_{K}.                                                         
	\eeq	
	Here $\|\cdotp\|_{K}$ means the supremum norm over $K$. 
	
By almost optimal, or "near-best" approximants $p_{n}$ to a function $f \in A(K)$ we call polynomials, such that
$$
\|f - p_{n}\|_{K} \leq C E_{n}(f), \qquad n \in \N
$$  
with a fixed $C>1$.
	 
		One of the starting points for the study of almost optimal polynomials traces back to the work by Shirokov \cite{ShirokovDocl}. It was pointed out that for a $D$-approximation, first introduced by Dzjadyk (\cite{Dz1}, \cite{Dz2}), the rate of decay of the error strictly inside a continuum can be $c n^{-M}$, with an arbitrary fixed $M$.
		
		Later on, the existence of almost optimal polynomials, possible rate of convergence inside $K$ and its dependence on the geometric properties of $K$ were studied in works by Saff and Totik \cite{AppPieceFunct}, \cite{BehaviorBestUnApp}, Maimeskul \cite{Maimeskul1}, \cite{Maimeskul}, Shirokov and Totik \cite{AppOnTheBoundAndInside}, Kryvonos \cite{Kryv2}, \cite{Kryv1}. 
	
	   In particular, it was proved by Saff and Totik \cite{BehaviorBestUnApp} that if $K$ is bounded by an analytic curve, then there are almost optimal approximants $p_{n}$ such that, for every compact set $E \subset \ir K$, 
	   $$
	   |f(z) - p_{n}(z)| \leq C e^{-\beta n} E_{n}(f), \;\;\; z \in E,
	   $$
	where  $\beta >0$, $C\geq1$ are independent of $f$. In \cite{Kryv1} it was shown that the same behavior holds for almost optimal harmonic polynomials. 
	
	Our primary motivation for this paper is to extend the family of compact sets for which the geometric convergence of "near-best" approximants is achieved whenever it is permitted.
	Using the rapidly decreasing polynomials introduced by Shirokov \cite{Shirokov}, we prove the following.

	\begin{theorem} \label{th1}
		Let $K \subset \Cc$ be a compact convex set with nonempty interior $\ir K$. Then for any $f \in A(K)$ there exists a sequence of "near-best" polynomials $\{p_{n}\}_{1}^{\infty}$, such that
		$$
		\underset{n \ra \infty}{\overline{\lim}} \|f - p_{n}\|_{E} \;e^{cn} =0
		$$
		holds for any compact set $E \subset \ir K$, with $c=c(E)$.
	\end{theorem}

   It turns out, if a compact set has an external angle smaller than $\pi$, then geometric convergence of almost optimal polynomials might be impossible. 
   For a disc with a sector of angle $\alpha \pi$ removed, $\alpha < 1$, this was shown in \cite{Maimeskul} and \cite{AppOnTheBoundAndInside}. Nevertheless, the value of an external angle alone in general does not characterize the behavior of almost optimal polynomials, as we demonstrate in the next example.

 For $0<\alpha<1$ and some $R>0$, consider a compact set $K$ whose boundary consists of the arcs $\{re^{\pm i \alpha \pi/2}, 0 \leq r \leq R\} \cup \{re^{i (\pi \pm \alpha \pi/2)}, 0 \leq r \leq R\} \cup \{Re^{ i \phi}, -\alpha \pi/2 \leq \phi \leq \alpha \pi/2\} \cup \{Re^{ i (\pi + \phi)}, -\alpha \pi/2 \leq \phi \leq \alpha \pi/2\}$. For such compact set we obtain the next result.

\begin{theorem} \label{AngleDom}
	Consider a function $f \in A(K)$, that is not analytic at $z=0$. Then, if $\alpha \leq \beta:=1 - \alpha$, there exists a sequence of "near-best" polynomials that have geometric rate of convergence in the interior points of $K$. If $\alpha > \beta$, then the geometric rate of convergence in the interior of $K$ is impossible.
\end{theorem}
	
	Rapidly decreasing polynomials (\ref{pol}) can also be applied for enhancing rational approximation on convex sets. Gopal and Trefethen in \cite{Tr} showed that for a function $f$ with corner singularities approximation by rational functions can achieve root-exponential convergence. We show that approximants can be modified so that at points of analyticity of $f$ they achieve geometric convergence.

	\begin{theorem} \label{polygon}
		Let $K$ be a convex polygon with corners $\omega_{1}, ..., \omega_{k}$ and let f be an analytic function in $K$ that is analytic on the interior of each side segment and can be analytically continued to a disc near each $\omega_{i}$ with a slit along the exterior bisector there. Assume $f$ satisfies $f(z) - f(\omega_{i}) = O(|z - \omega_{i}|^{\delta})$, $z \ra \omega_{i}$ for each $i$ for some $\delta >0$.
		There exist rational functions $\{r_{n}\}$ of degree $n$, $1 \leq n < \infty$ satisfying 
		$$
		\|f - r_{n}\|_{K} = O(e^{-C\sqrt{n}}),
		$$
		for some $C>0$, and 
		$$
		\underset{n \ra \infty}{\overline{\lim}} \|f - r_{n}\|_{E} e^{cn}= 0,
		$$
		for an arbitrary compact subset $E \in \textnormal{Int} K$ with $c=c(E)>0$ independent of $n$.
	\end{theorem}
	
	The paper is organized as follows. In Section 2 we introduce fast-decreasing polynomials and discuss their properties. Sections 3 and 4 contain the proofs of Theorem \ref{th1} and Theorem \ref{AngleDom}, respectively. Finally, in Section 5 we consider an application of polynomials (\ref{pol}) for rational approximation.

	\section{Notation and preliminaries}
	 For $a>0$ and $b>0$ we will use the notation $a\preccurlyeq b$ if $a\leqslant cb$, with some constant $c>0$. The expression $a\asymp b$ means $a\preccurlyeq b$ and $b\preccurlyeq a$.
	
	The open disc of radius $\delta$ centered at $z$ is denoted by $D(z,\delta)$. 
	
	Let $d(A,B)$ be the distance between $A\subset \mathbb{C}$ and $B\subset\mathbb{C}$. Further, for $A\subset \mathbb{C}$ and $\delta>0$, we set
	$$
	A_{\delta}:= \{\zeta: d(\zeta, A)< \delta\}.
	$$
	
	Consider a conformal mapping
	$\Phi: \Omega \rightarrow \Delta := \{\omega: |\omega|>1\}$, normalized in such a way that $\Phi(\infty)=\infty$, $\Phi'(\infty)>0$,
	and denote $\Psi:= \Phi^{-1}$.
	By $\widetilde{\Omega}$ we denote compactification of the domain  $\Omega$ by prime ends in the Caratheodory sense. Denote by $|\Zz|$ the impression of a prime end $\Zz \in \widetilde{\Omega}$.
	
	For $u>0$ define
	$$
	\mathcal{L}_{u}:=\{\zeta \in \Omega: |\Phi(\zeta)|=1+u\};
	$$ $$
	\rho_{u}(z):=d (z,\mathcal{L}_{u}); 
	$$
	$$
	 \widetilde{\zeta}_{u}:= \Psi[(1+u)\Phi(\mathfrak{Z})], \;\;\; \mathfrak{Z} \in \widetilde{\Omega}.
	$$ 
	\begin{definition} \label{definition}
		We say that a compact set $K$ with a rectifiable boundary $L$ belongs to the class $B_{k}$ for some $k \in \N$ if for all $n \in \N$ the boundary $L$ has the following properties:
		
		(i)  $\forall \zeta, z \in L$ the linear measure of $s(\zeta,z)$ of the set $L \cap D(z; |z - \zeta|)$ satisfies
		$$
		s(\zeta ,z) \preccurlyeq |\zeta - z|;
		$$
		
		(ii) the boundary $L$ can be represented, including its folds, as $L = \cup_{j=1}^{k} L_{j}$, with $k \in \N$ being a finite number, so that $\forall \zeta, z \in L_{j}$, $j = 1,...,k,$ the relations 
		$$
		|z - \widetilde{z}| \asymp \rho^{(j)}_{1/n}(z),
		$$
		$$
		|\zeta - \widetilde{\zeta}|^{k} \preccurlyeq |\widetilde{\zeta} -z|^{k-1} |\widetilde{z} -z|, 
		$$
		hold, where $\rho^{(j)}_{1/n}(z)$ is the distance of $z$ to the part of the $n^{th}$ level line
		$$
		\widetilde{L}_{j}:=\Psi((1+\frac{1}{n})\Phi(L_{j})),
		$$
		and the constants in $\preccurlyeq$ and $\asymp$ do not depend on $n$, $\zeta$ and $z$.
	\end{definition}
	
	Following Shirokov \cite{Shirokov}, we introduce polynomials $R_{n}$ and estimate them on the corresponding squares.
	
    	\begin{lemma} \label{improving}
    	Let $S_{l}$ be a square $\{z = x+ iy: -2l \leqslant x \leqslant 0, |y|\leqslant l\}$.
    	There exists a polynomial $R_{n}$ of degree $3n$ of the form $R_{n}(z) = 1+ zp(z)$ with the property
    	\beq \label{ineq}
    	|R_{n}(z)| \leqslant 1, \;\;\; z \in S_{l}; \;\;\; |R_{n}(z)| \leqslant C, \;\;\; z \in S_{l+c/n},
    	\eeq
    	where $C, c$ do not depend on $n$, and
    	there exists a positive function $c(\varepsilon)$, $0<\varepsilon<1$, such that if $z \in S_{l}$, $d(z, \partial S_{l}) \geqslant \varepsilon l$, then
    	$$
    	|R_{n}(z)| \leqslant Ce^{-c(\varepsilon)n}.
    	$$ 
    	
    \end{lemma} 
    
    \begin{proof}
    	Clearly, it is enough to prove the statement for the square $S_{1/3} = \{z = x+ iy: -\frac{2}{3} \leqslant x \leqslant 0, |y|\leqslant \frac{1}{3}\}$. 
    	Let 
    	$$
    	r(z) = \big(1 + z\big)\big(1 + z^{2}\big),
    	$$
    	\beq \label{pol}
    	R_{n}(z) = (r(z))^{n} = 1+ z p(z),
    	\eeq
    	where $p(z)$ is a polynomial of degree $3n-1$.
    	
    	To estimate the value of $r(z)$ on the sides of $ S_{1/3 + c/n}$, we consider a several cases:
    	
    	1) For the sides of the square $z = - \frac{2}{3}+ iv$, $|v|\leqslant \frac{1}{3}$ and $z = - u \pm \frac{i}{3}, 0 \leqslant u \leqslant \frac{2}{3}$ it was shown in \cite{Shirokov} that
    	$|r(z)| < 1$ at those points, therefore, it remains true for the corresponding sides of $S_{1/3+c/n}$ with $n$ large enough.
    	
    	2) Let $z = \delta + iv$, $|v|\leqslant \frac{1}{3} + \delta$.
    	In this case 
    	\begin{align*}
    		|r(z)|^{2} =& |1+ z|^{2} |1+ z^{2}|^{2} = \big[ \big(1+ \delta \big)^{2} + v^{2} \big] \big[ \big(1+ \delta^{2}  - v^{2}\big)^{2} + \big(2 \delta v\big)^{2}\big] \\
    		=&  \big[  \big(1+\delta^{2} + v^{2}\big) + 2\delta\big] \big[ \big(1+ \delta^{2}  - v^{2}\big)^{2} + \big( 2 \delta v \big)^{2}\big] \\
    		\leqslant& \big(1+ \delta^{2}  - v^{2}\big) \big(\big(1+ \delta^{2}\big)^{2}  - v^{4}\big) + c_{1} \delta \leqslant 1+  c_{2} \delta
    	\end{align*}
    	Thus, for $\delta:=\frac{c}{n}$, and $z = \frac{c}{n} + iv$, $|v|\leqslant \frac{1}{3} + \frac{c}{n}$ we have 
    	$$
    	|R_{n}(z)| \leqslant \big(1+  \frac{c_{3}}{n}\big)^{n} \leqslant C, 
    	$$
    	which completes the proof.
   
    	\vspace{\baselineskip}
    \end{proof}
              Next, we describe the construction of rapidly decreasing polynomials for a convex set $K$.
              For $\zeta \in \Omega \cap K_{\varepsilon}$, $\varepsilon>0$, consider the corresponding $S_{\zeta}$, a square with a side length $l= 2 \textnormal{diam} K_{\varepsilon}$, built as follows. 
              Take a point $\zeta_{0}$, which is the closest to the $\zeta$. Construct the line $t$, orthogonal to the segment $[\zeta_{0}, \zeta]$ and passing through $\zeta$. Since $K$ is convex, it lies in one of the half-spaces bounded by $t$. We let $\zeta$ be the middle point of one of the sides of $S_{\zeta}$, and $S_{\zeta}$ choose to be lying in the half-space containing $K$. By $\Theta(\zeta)$ denote an angle between the $[\zeta_{0}, \zeta]$ and positive $x$ axes.    
              
              For $\zeta \in K \cap \Omega_{\delta}$ the construction is analogous, we choose $\zeta_{0}$ being the closest point to $\zeta$ on $\partial K$ and consider the corresponding square $S_{\zeta}$. In this case some points of $K$ will lie outside of $S_{\zeta}$, at distance at most $\widetilde{c}\delta$ from the side of $S_{\zeta}$ that passes through $\zeta$.
              
              Finally, set
              \beq \label{fastPol}
              R_{n}(z,\zeta) := R_{[n/3]}(e^{-i\Theta(\zeta)}(z-\zeta)),
              \eeq
              where $\zeta \in \Omega_{\delta} \cap K_{\varepsilon}$, $z \in K$.
              
            \section{Proof of Theorem \ref{th1}}
            \begin{proof}
            First, consider
            $$f_{n}(z):= \frac{f(z)-p^{*}_{n}(z)}{E_{n}(f)},$$
            where $p^{*}_{n}$ are polynomials of best approximation to the function $f$ on $K$.
            
            By $P_{n}(z)$ we denote the polynomial 
            $$
            P_{n}(z):= \frac{1}{2 \pi i } \int_{L} f_{n}(\zeta)Q_{n}(\zeta, z) d \zeta,
            $$
            with
            $$
            Q_{n}(\zeta,z):= \frac{1 - R_{[n/2]}(z,\zeta)}{\zeta -z} + R_{[n/2]}(z,\zeta) K_{[n/2]}(\zeta, z), 
            $$
            where $ R_{n}(z,\zeta)$ are polynomials (\ref{fastPol}), and 
            $
            K_{n}(\zeta, z) := K_{0,1,k,[\varepsilon n]}(\zeta,z)
            $
            is Dzjadyk's polynomial kernel (see \cite{ConstrTheory}, Chapter 3) with sufficiently small $\varepsilon$.
            
            Consider another family of functions
            $$
            \widetilde{f}_{n}(z) = \frac{1}{2 \pi  i} \int_{L} f_{n}(\zeta) \frac{1}{\zeta - z^{*}} d \zeta,
            $$
            where $z^{*}$ coincides with $z$ if $z \in \textnormal{Int} K$, and $z^{*} \in \Cc \setminus K$ is at the distance $n^{-4}$ from $z$ in case $z \in L$.
            
            Then we have 
            $$
            I(z):= \widetilde{f}_{n}(z) - P_{n}(z) = 
            \frac{1}{2 \pi i } \int_{L} f_{n}(\zeta)  \bigg(\frac{1}{\zeta - z^{*}} - Q_{n}(\zeta,z)\bigg) d \zeta.
            $$

            If $z \in \cup_{\zeta \in L}  \overline{D(\zeta, \rho_{1/n}(\zeta))}$, then by $\zeta_{0}$ we denote a point on $L$ which is the center of a disc $D(\zeta_{0},\rho_{1/n}(\zeta_{0}))$ of the largest radius $\rho_{1/n}(\zeta)$ containing $z$. Let $D_{0}:= \overline{D(\zeta_{0},2\rho_{1/n}(\zeta_{0}))}$. Then one can see that the disc $D_{0}$ contains all the points $\zeta \in L$, such that $z \in \overline{D(\zeta,\rho_{1/n}(\zeta))}$. Denote by $\gamma, \sigma$ and $L_{j}'$ the arcs $\gamma:=L \cap D_{0}$, $\sigma:=\textnormal{Int}K \cap \partial D_{0}$ and $L_{j}':=L_{j} \setminus \gamma$. (here the boundary $L$ consists of the union of subarcs $L_{j}$ from the definition of the class $B_{k}$).

            Now, since $\rho_{1/n}(\zeta) \asymp \rho_{1/n}(z)$, $\forall$ $z \in L$, $\zeta \in L \cap D(z, \rho_{1/n}(z))$ (\cite{ConstrTheory}, Chapter 3), we get
            \begin{align*} 
            	|I_{\gamma}| =& \bigg| \int_{\gamma}  f_{n}(\zeta)  \bigg(\frac{1}{\zeta - z^{*}} - Q_{n}(\zeta,z)\bigg) d \zeta \bigg|\\
            	=& \bigg| - \int_{\sigma} f_{n}(\zeta) \bigg\{ \frac{z^{*} - z}{(\zeta - z^{*}) (\zeta - z)}  +  R_{[n/2]}(z,\zeta) \bigg( K_{[n/2]}(\zeta,z) - \frac{1}{\zeta -z} \bigg)\bigg\} d \zeta\bigg| \\
            	\preccurlyeq& n^{-4} \frac{1}{\rho_{n}(\zeta_{0})}  + \frac{1}{\rho_{n}(\zeta_{0})} \rho_{n}(\zeta_{0}) \preccurlyeq 1, 
            \end{align*}
            where for the first term in the inequality we used the fact that $\rho_{n}(\zeta) \succcurlyeq 1/n^{2}$,  and for the second one the bound
            $$
            \bigg| \frac{1}{\zeta -z} - K_{n}(\zeta, z) \bigg| \preccurlyeq \frac{\rho^{(j)}_{1/n}(z^{(j)})}{|\zeta -z| (|\zeta - z| + \rho^{(j)}_{1/n}(z^{(j)}))},
            $$ 
            where $z^{(j)}$ is a point of $L_{j}$ which is the closest to $z$ (see \cite{ConstrTheory}, Chapters 2 and 3, respectively).
            \vspace{\baselineskip}
            
            Now we estimate the integral along $L_{j}'$. Let 
            $$
            r:= r_{j} = \frac{1}{2}(|z - z^{(j)}| + \rho_{1/n}(z^{(j)}) )
            $$
            and consider circles $O_{s}$ centered at $z$ with radius $2^{s}r$, $s = \overline{0,N}$, where $N$ is chosen such that
            $$
            O_{N-1} \cap L'_{j} \neq \emptyset; \;\;\; O_{N} \cap L'_{j} = \emptyset
            $$ 
            By $\tau_{s}$ denote the parts of the arc $L'_{j}$ which lies between the circles $O_{s}$ and $O_{s+1}$, $s=\overline{0,N-1}$. According to the property of the set of class $B_{k}$, we have $\textnormal{mes} \tau_{s} \preccurlyeq 2^{s} r$.
            
            Thus, we obtain
            \begin{align*}
            	&\bigg| \frac{1}{2 \pi i } \int_{L'_{j}} f_{n}(\zeta)  \bigg(\frac{1}{\zeta - z^{*}} - Q_{n}(\zeta,z)\bigg) d \zeta \bigg| \\
            	\leqslant&  \sum_{s=0}^{N-1} \int_{\tau_{s}} |f_{n}(\zeta)| \bigg| \frac{1}{\zeta - z^{*}} - Q_{n}(\zeta,z)\bigg| |d\zeta| \\
            	\preccurlyeq&  \sum_{s=0}^{N-1} \int_{\tau_{s}} |f_{n}(\zeta)| \bigg|   R_{[n/2]}(z, \zeta)\bigg( K_{[n/2]}(\zeta,z) - \frac{1}{\zeta -z} \bigg)\bigg| |d\zeta| \\
            	\preccurlyeq&  \sum_{s=0}^{N-1} \frac{1}{2^{s}r} \bigg( \frac{\rho^{(j)}_{1/n}(z^{(j)})}{2^{s}r}\bigg) 2^{s}r \preccurlyeq \frac{\rho^{(j)}_{1/n}(z^{(j)})}{r} \sum_{s=0}^{\infty} \frac{1}{2^{s}} \preccurlyeq 1.
            \end{align*}
            In addition, it is easy to see, due to the properties $R_{n}(z,\zeta)$, that for any compact set $E \subset \ir K$, 
            $$
            \|\widetilde{f}_{n} - P_{n}\|_{E} \preccurlyeq e^{-c n},
            $$
            where $c=c(E)$.
            
            Consequently, 
            $$
            p_{n}(z) := E_{n}(f) P_{n}(z) +  p^{*}_{n}(z)  
            $$
            is the desired sequence of "near-best" approximants. 
           \end{proof}
          
          \vspace{\baselineskip}

          Modifying the construction from \cite{Kryv1} with polynomials (\ref{fastPol}) for approximating harmonic functions, we also obtain
          	
          \begin{theorem} \label{th2}
          	Let $K \subset \Cc$ be a compact convex set with nonempty interior $\ir K$, $u \in H(K)$. Then there exists a sequence of "near-best" harmonic polynomials $\{h_{n}\}_{1}^{\infty}$, such that
          	$$
          	\underset{n \ra \infty}{\overline{\lim}} \|u - h_{n}\|_{E} \;e^{cn} =0
          	$$
          	holds for any compact set $E \subset \ir K$, with $c=c(E)$.
          \end{theorem}

      \section{Proof of Theorem \ref{AngleDom}}
      
      $\bullet$\textbf{ Case} $\alpha \leq \beta$
      
      	Notice that the image of points on the boundary of $K$ under the mapping $w = z^{2}$ is a set  $\partial K':=\{re^{\pm i \alpha \pi}, 0 \leq r \leq R^{2}\} \cup \{R^{2}e^{ i \phi}, -\alpha \pi \leq \phi \leq \alpha \pi\}$, and all the points of the interior of $K$ will be mapped to the interior points of $K'$. Notice also that $K'$ is a convex set, therefore we apply the construction above and define polynomials 
      
      $$
      R_{n}(z,\zeta) := R_{n}(e^{-i\Theta(\zeta^{2})}(z^{2}-\zeta^{2})),
      $$
      where $\zeta \in L:=\partial K$, $z \in K$.
      
      The set $K$ belongs to the class $B^{0}_{k}$, since it belongs to the class $B_{k}$ and satisfies the following (\cite{ConstrTheory}, Chapter 3):\\
      i) $\rho^{(1)}_{1/n} (Z_{0}) \asymp \rho^{(2)}_{1/n}(Z_{1})$, where $Z_{0}, Z_{1}$ are prime ends, such that $|Z_{i}| =0, \; i=1,2$.\\
      ii) $\forall$ $z \in L$ for each point $\zeta \in L \cap D(z, \rho_{1/n}(z))$
      $$
      \rho_{1/n}(\zeta) \asymp \rho_{1/n}(z).
      $$	
      
      Therefore, using properties i), ii), we replicate the construction of Theorem \ref{th1} and obtain the polynomials with the desired properties.

      	$\bullet$\textbf{ Case} $\alpha > \beta$
      
      First we prove an auxiliary lemma
      \begin{lemma} \label{EstimPol}
      	Suppose a polynomial of degree $n$ satisfies
      	$$
      	\|Q\|_{K} \leqslant C 
      	$$
      	and for some open set $V \subset \ir K$ 
      	$$
      	\|Q\|_{\overline{V}} \leqslant C e^{-cn}
      	$$
      	Then, for $|z|\leqslant \varepsilon$, $\varepsilon = \varepsilon(\widetilde{\alpha}, \beta)$, with $\beta<\widetilde{\alpha} < \alpha$, the inequality 
      	$$
      	|Q(z)| \leqslant C \exp(-c'n|z|^{1/\widetilde{\alpha}} + c''n|z|^{1/\beta}),
      	$$
      	holds.
      \end{lemma}
      
      \begin{proof} Choose $\alpha'$ such that $\beta < \alpha' < \alpha$ which satisfies in addition $(\alpha')^{2}>\alpha \beta$, and consider a cone which is essentially a shifted to the left (by some $w >0$) cone with $\partial M:= \{re^{\pm i \alpha' \pi/2}, 0 \leq r \leq r_{0}\} \cup \{r_{0}e^{ i \phi}, -\alpha' \pi/2 \leq \phi \leq \alpha' \pi/2\}$, where $r_{0} < R$.
      	Since the set $V$ in the statement of the lemma contains some disc, we may assume that the estimate 
      	$$
      	|Q(z)| \leqslant C e^{-cn}
      	$$
      	holds for all $z$ such that  $z + w = r_{0}e^{ i \phi}, -\alpha' \pi/2 \leq \phi \leq \alpha' \pi/2$.
      	(in other words, for the points of the arc of the cone $M'$ obtained by shifting $M$ as described above). By $\gamma$ we denote the linear subarcs of the boundary of $M'$.
      	
      	The mapping $\Phi$ satisfies
      	$$
      	|\Phi(z)| - 1 \leqslant C' |z|^{1/\beta}, \;\;\; |z|<1,
      	$$ 
      	and, according to the Bernstein's inequality, for $z \in \gamma \setminus K$ 
      	$$
      	|Q(z)| \leqslant C |\Phi(z)|^{n} \leqslant C (1 + C'|z|^{1/\beta})^{n} \leqslant C \exp(C'n|z|^{1/\beta}).
      	$$   
      	Also we see that for all $z \in \gamma \setminus K$
      	$$
      	|z| \leqslant \widetilde{c} w,
      	$$
      	where $\widetilde{c} = \widetilde{c}(\alpha, \alpha')$
      	
      	Then the estimate 
      	\begin{align}
      		\log|Q(z)| &\leqslant \log(C\exp(C'n(\widetilde{c} w)^{1/\beta})) m_{\gamma}(z) + \log( C e^{-cn})m_{\partial M' \setminus \gamma}(z) \\&= C'\widetilde{c}^{1/\beta}n  w^{1/\beta} m_{\gamma}(z)  - cn \; m_{\partial M' \setminus \gamma}(z) + \log(C),
      	\end{align}
      	holds, where $m_{\gamma}(z)$ and $m_{\partial M' \setminus \gamma}(z)$ are harmonic measures of the corresponding arcs at point $z$.
      	
      	We will estimate $m_{\partial M' \setminus \gamma}(z)$ from below assuming $|z| \leqslant \frac{1}{2} d(0, \partial M')$ (we also have here $d(0, \partial M') \asymp w$).
      	For this we estimate the module of family $\Gamma$ of curves separating $z$ from the arc $m_{\partial M' \setminus \gamma}(z)$. We refer the reader to (\cite{Blatt}, Appendix A) for the definition and main properties of the module of curve family that we use further.
      	Consider a function $\psi$ which maps $M'$ conformally to the unit disc and satisfies $\psi(z) = 0$. In this case $\psi(\Gamma)$ is a family of curves separating $0$ from the subarc $J$ of the unit circle that satisfies $|J| = 2 \pi m_{\partial M' \setminus \gamma}(z)$.
      	As it is shown in the Example 1.11 (\cite{Blatt}, Appendix A),
      	$$
      	\frac{1}{\pi} \log\frac{2}{|J|} \leqslant m(\psi(\Gamma)) \leqslant 2 + \frac{1}{\pi} \log \frac{4}{|J|}
      	$$
      	
      	On the other hand, we estimate $m(\Gamma)$ from above as follows. 
      	Let $d:=d(z, \partial M')$.
      	Consider the metric
      	
      	\begin{gather}
      		\rho(\zeta)=\begin{cases}                                                              
      			\frac{1}{\alpha' \pi |\zeta -z|},&\text{if $e^{- \alpha' \pi} d\leqslant|\zeta - z|\leqslant c$, \;\;$\zeta \in M'$}\\                             
      			0,&\text{otherwise}                   
      		\end{cases}                                                                   
      	\end{gather}
      	
      	Then for any $\gamma \in \Gamma$ we have 
      	$$
      	\int_{\gamma} \rho(\zeta) |d\zeta| \geqslant 1,
      	$$
      	and
      	$$
      	m(\Gamma) \leqslant \int_{\Cc} \rho^{2}(\zeta) dm(\zeta) \leqslant \frac{\alpha}{(\alpha')^{2} \pi}\ln\bigg(\frac{c_{1}}{d}\bigg) + c_{2},
      	$$
      	
      	Thus, combining the estimates we obtain
      	$$
      	\frac{1}{\pi} \log\frac{2}{|J|} \leqslant m(\psi(\Gamma)) = m(\Gamma) = \frac{\alpha}{(\alpha')^{2} \pi}\ln\bigg(\frac{c_{1}}{d}\bigg) + c_{2}
      	$$
      	which implies
      	$$
      	|J| \curlyeqsucc d^{\alpha/(\alpha')^{2}} \curlyeqsucc w^{\alpha/(\alpha')^{2}}.
      	$$
      	
      	Consequently, together with estimate $m_{\gamma}(z) \leqslant 1$ we obtain
      	$$
      	\log|Q(z)| \leqslant c_{1} n  w^{1/\beta}  - c_{2}n \; w^{\alpha/(\alpha')^{2}} + \log(C),
      	$$
      	which finishes the proof of the lemma.
      \end{proof}
      
      To finish the proof of the theorem, assume that there exists a sequence of almost optimal polynomials $P_{n}$ satisfying
      $$
      \|f - P_{n}\|_{\overline{V}} \leq C E_{n}(f) e^{-cn}
      $$
      for some open $V \subset \ir K$.
      Then, by Lemma \ref{EstimPol},
      $$
      |P_{n}(z) - P_{n-1}(z)| \leqslant C \exp(-c'n|z|^{\alpha/(\alpha')^{2}} + c''n|z|^{1/\beta}),  \;\;\;\;\;|z|\leqslant \varepsilon.
      $$
      Since $(\alpha')^{2} > \alpha \beta$, there exists some $\varepsilon' <  \varepsilon$ such that
      
      $$
      |P_{n}(z) - P_{n-1}(z)| \leqslant C \exp(-c_{3}n),  \;\;\;\;\;|z|\leqslant \varepsilon'.
      $$
      
      Thus, $P_{n}$ converge to some analytic function in $|z| \leqslant \varepsilon'$, that is a contradiction, since $f$ is not analytic at $z=0$. 
      
      \section{Approximation with rational functions}
   
  Theorem \ref{polygon} follows directly from the next result.
\begin{theorem}
	Consider $A_{\Theta}:=\{z \in \Cc: |z| < 1, -\Theta < \arg z < \Theta \}$ and let $\Theta \in (0, \pi/2)$ be fixed. Let $f$ be a bounded analytic function in the slit disc $A_{\pi}$ that satisfies $f(z) = O(|z|^{\delta})$ as $z \ra 0$ for some $\delta >0$.  Then, for $\rho \in (0,1)$, there exist rational functions $\{r_{n}\}$, $1 \leq n < \infty$, such that 
	$$
	\|f - r_{n}\|_{\overline{M}} = O(e^{-C\sqrt{n}}), 
	$$   
	as $n \ra \infty$ for some $C>0$, where $M = \rho A_{\Theta}$,
	and
	$$
	\underset{n \ra \infty}{\overline{\lim}} \|f - r_{n}\|_{E} \;e^{cn} =0,
	$$
	for any compact $E \subset M$, with $c=c(E)$.
\end{theorem}
 
\begin{proof}
	Let 
	\begin{equation*}
		\phi(z) := \prod^{n-1}_{j=0} (z - \alpha_{j}) / \prod^{n-1}_{j=0} (z - \beta_{j}),
	\end{equation*}
	where we define poles $\beta_{j}$ by
	$$
	\beta_{j}:=-e^{\ \sigma j/\sqrt{n}}, \;\;\; 0\leq j \leq n-1,
	$$ 
	with some fixed $\sigma>0$, and interpolation points $\alpha_{j}$ by 
	$$
	\alpha_{0} =0, \;\;\; \alpha_{j} = -\beta_{j}, \;\;\; 1 \leq j \leq n-1.
	$$
	
	Let $\Gamma$ be a boundary of the slit disc $A_{\pi}$. We will approximate the Cauchy integral of the function $f$ over the contour $\Gamma$.
	It is enough to approximate the integral over $[-1,0]$ only, since the integral over the circular part, by Runge's theorem \cite{Runge1}, can be approximated with geometric rate on $\overline{M}$.
	Set
	$$
	r_{n}(z):= \frac{1}{2 \pi} \int_{[-1,0]} q_{[\frac{n}{2}]}(\zeta,z) f(\zeta) d \zeta, 
	$$
	where
	$$
	q_{n}(\zeta,z) := \frac{1}{\phi(\zeta)}\frac{\phi(\zeta) - \phi(z)}{\zeta-z} + \frac{\phi(z)}{\phi(\zeta)} \frac{1- R_{n}(\zeta,z)}{\zeta-z}
	$$
    is a rational function of order at most $2n$.

For $\varepsilon = \exp(- \sigma (n-1)/\sqrt{n})$, split $[-1,0]$ into two pieces,
$$
\Gamma_{\varepsilon}:= \{\zeta \in \Gamma: |\zeta|< \varepsilon\}, \;\;\; 
\Gamma_{1} = [-1,0] \setminus \Gamma_{\varepsilon}
$$

Then we get
\begin{align*}
f(z) - r_{2n}(z) &= \frac{1}{2 \pi} \int_{\Gamma_{\varepsilon}} \frac{\phi(z)}{\phi(\zeta)} \frac{f(\zeta)}{\zeta -z} R_{n}(\zeta,z)  d\zeta \\
& + \frac{1}{2 \pi} \int_{\Gamma_{1}} \frac{\phi(z)}{\phi(\zeta)} \frac{f(\zeta)}{\zeta -z} R_{n}(\zeta,z)  d\zeta =  I_{\varepsilon} + I_{1}.
\end{align*}
To bound $I_{\varepsilon}$, first notice that for $\zeta \in \Gamma_{\varepsilon}$, we have $|z| \leq |\zeta - z|$.
Also the estimates $|z-\alpha_{0}|/|z-\beta_{0}| \leq |z|$ and $|z - \alpha_{j}|/|z - \beta_{j}|$, $j \geq 1$ imply
\beq \label{phiBound}
|\phi(z)| \leq |z|, \;\;\; z \in A_{\Theta}. 
\eeq
Thus,
$$
\frac{|\phi(z)|}{|\zeta -z|} \leq 1.
$$
To bound $f(t)/\phi(t)$, over $\Gamma_{\varepsilon}$, first observe that
$$
|\phi(\zeta)| \geq |\zeta|, \;\;\; \zeta \in \Gamma_{\varepsilon},
$$
since $|\zeta - \alpha_{0}|/ |\zeta - \beta_{0}| \geq |\zeta|$, and $|\zeta - \alpha_{j}|/ |\zeta - \beta_{j}| \geq 1$ for $j \geq 1$.
Therefore, taking into account the behavior of $f(z)$ as $z \ra 0$, it's enough to bound the integral
$$
\int_{0}^{\varepsilon} \frac{t^{\delta}}{t} dt =  \frac{\varepsilon^{\delta}}{\delta}
$$
Since $\varepsilon = \exp(- \sigma (n-1)/\sqrt{n})$, the integral is of order $\exp(-\sigma \delta \sqrt{n})$.

By (\ref{phiBound}), the integral over $I_{1}$ is root-exponentially small for $z \in A_{\Theta}$ with $|z| < 2 \varepsilon$.
For $z \in A_{\Theta}$ with $|z| \geq 2 \varepsilon$, it is enough to show that order $\sqrt{n}$ of factors $|z - \alpha_{j}|/ |z -\beta_{j}|$ is bounded by some constant $A <1$ that is independent of $z$. 
For each $z$, choose the factors satisfying $|z|/2<\alpha_{j}<|z|$. It's easy to see that the number of these factors grows in proportion to $\sqrt{n}$, $n \ra \infty$ and that each factor is bounded by some constant $A<1$, depending on $\Theta$ but independent of $z$.
Moreover, for the points $z \in M$ the approximants converge with geometric rate, due to the properties (\ref{ineq}) of the polynomials $R_{n}$.

\end{proof}

	\end{document}